\documentclass[10pt]{amsart}
\usepackage{amsmath,graphicx,amssymb,amsthm}
\usepackage{color}
\def\A{\mathcal{A}}

\def\H{\mathcal H}

\def\L{\mathcal L}

\def\P{\mathcal P}

\def\X{\mathcal X}

\def\amslatex{$\mathcal{A}\kern-.1667em\lower.5ex\hbox{$\mathcal{M}$}\kern-.125em\mathcal{S}$-\LaTeX}

\newtheorem{set}{set}[section]
\newtheorem{Corollary}[set]{Corollary}
\newtheorem{Definition}[set]{Definition}

\newtheorem{Remark}[set]{Remark}
\newtheorem{Theorem}[set]{Theorem}
\newcommand{\define}{\mathrel{\hbox{$\equiv$\hskip -.90em \lower .47ex \hbox{$\leftharpoondown$}}}}
\newcommand{\enifed}{\mathrel{\hbox{$\equiv$\hskip -.90em \lower .47ex \hbox{$\rightharpoondown$}}}}

\begin{document}
\title {\bf   Non-Gaussian Random Matrix Models for Two-faced Families of Random Variables Having Bi-free Central Limit Distributions }
\author{Mingchu Gao}
\address{School of Mathematics and Information Science,
Baoji University of Arts and Sciences,
Baoji, Shaanxi 721013,
China; and
Department of Mathematics,
Louisiana College,
Pineville, LA 71359, USA}
\email[Mingchu Gao]{mingchu.gao@lacollege.edu}

\date{}
\begin{abstract} In this paper, we construct random two-faced families of  matrices with non-Gaussian entries to approximate a two-faced family of random variables having a bi-free central limit distribution. We prove that, under modest conditions weaker than independence, a family of random two-faced family of  matrices with non-Gaussian entries is  asymptotically bi-free from a two-faced family  of constant diagonal matrices.
\end{abstract}
\maketitle
{\bf AMS Mathematics Subject Classification (2010)} 46L54.

{\bf Key words and phrases} Bi-free central limit distributions, Non-Gaussian random matrices.
\section*{Introduction}
D. Voiculescu introduced the notion of {\sl pairs of faces of random variables} in \cite{DV}, initiating a new research area in free probability, {\sl free probability for pairs of faces}, or {\sl bi-free probability}. Since then, the theory has quickly developed by generalizing ideas and results from free probability to this new setting.   For instance, Voiculescu \cite{DV}  determined bi-free central limit distributions. \cite{CNS2} and \cite{CNS1} developed the combinatorial aspect of the theory and bi-free probability with amalgamation. In addition, the notion of bi-free infinite divisible distributions for  commutative pairs of random variables was developed in \cite{GHM}, which was generalized later to the case of arbitrary pairs of random variables in \cite{MG}.

One of the most important achievements in free probability is Voiculescu's work in \cite{DV1}, which discovered a connection between free probability and random matrices: the phenomenon of freeness of random variables, initially coming from free products of group algebras (or operator algebras), or creation and annihilation operators on the full Fock space of a Hilbert space, could  be recognized as the asymptotic behavior of independent Gaussian  random matrices, as the matrix size approaches infinity. Since then, the study on the relations between free probability and random matrices has been a main theme in free probability.

P. Skoufranis \cite{PS} generalized Voiculescu's main results in \cite{DV1} to bi-free probability. Among other results in \cite{PS}, Skoufranis constructed bi-matrix models of Gaussian random variables for a bi-free central limit distribution of a commutative pair of random variables with a non-singular positive definite covariance matrix. He also  proved in \cite{PS} asymptotic bi-freeness  of  independent random pairs of  matrices of Gaussian random variables from pairs of constant matrices. There are non-Gaussian random matrix models for semicircle distributions and results of asymptotic freeness  of random Gaussian matrices from constant matrices in the literature (e. g., \cite{KD}, \cite{HP} and \cite{AGZ}). It is worth to notice that the family of limit random variables of random matrices in the literature is a free family of semicircular random variables. More generally, Remark 4.12 in \cite{PS} constructed a two-faced family of Gaussian random matrices to approximate a bi-free central limit distribution with a non-singular positive definite covariance matrix. As a consequence of Remark 4.12 in \cite{PS}, one can get a Gaussian random matrix model for a semicircular family  with a non-singular positive definite covariance matrix defined in 8.15 of \cite{NS}. In this paper, we construct  two-faced families of non-Gaussian random matrices to approximate in distribution to a two-faced family of random variables having a bi-free central limit distribution with an arbitrary positive definition covariance matrix. As a consequence of our approximation result, we get a random matrix model for an arbitrary semicircular family of a positive definite covariance matrix defined in 8.15 in \cite{PS}. We also prove asymptotic bi-freeness of a family of two-faced families of non-Gaussian random matrices from a two faced family of constant diagonal matrices provided that the families of random matrices satisfy certain conditions weaker than independence.

This paper is organized as follows. Besides this introductory section, this paper consists of two sections. In Section 1,  we recall necessary background materials on bi-free probability used in sequel. Section 2 is devoted to proving our results. We first construct  two faced families of non-Gaussian random matrix to approximate a two-faced family of random variables having a bi-free central limit distribution (Theorem 2.2). Note that a centered Gaussian family of random matrices must have a non-singular positive definite covariance matrix (Definition 4.1 in \cite{PS} and Definition 22.1 in \cite{NS}). Therefore, if a bi-free central limit distribution (or a semicircular family of random variables) has a  bi-matrix model consisting of Gaussian random matrices (or a Gaussian random matrix model), then the covariance matrix of the bi-free central limit distribution (or the semicircular family) must be non-singular.
Our non-Gaussian random  bi-matrix model provides a random bi-matrix model for any bi-free central limit distribution with a positive definite covariance matrix. As a consequence, we get a random matrix model for an arbitrary semicircular family defined in 8.15 in \cite{NS} (Remark 2.3). Theorem 4.11 in \cite{PS} shows that independent random pairs of Gaussian matrices are asymptotic bi-free. We prove that a family of random two-faced families of non-Gaussian matrices is asymptotic bi-free provided the families satisfies certain conditions ((2.3), (2.4), and (2.5)), which are weaker than the independence of the families (Theorem 2.4 and Remark 2.5). Finally, we prove asymptotic freeness of a family of non-Gaussian random matrices from a family of constant diagonal matrices, which is similar to, but different from, Theorem 2.1 in \cite{KD} (Theorem 2.6 and Remark 2.7). As a consequence of Theorem 2.6 of this paper and Theorem 4.13 in \cite{PS}, we get asymptotic bi-freeness of a family of two-faced families of non-Gaussian random matrices from a two-faced family of constant diagonal matrices (Corollary 2.8).
\section{Background on Bi-free Probability and Bi-matrix Models}
In this section we recall some background materials on combinatorial aspects of  bi-free probability, operator-valued bi-free probability, and bi-matrix models.  The reader is referred to \cite{DV}, \cite{CNS2}, \cite{CNS1}, and \cite{PS1} for details on bi-free probability.
\subsection{Combinatorics on Bi-free Probability}
 Let $\chi:\{1, 2, \cdots, n\}\rightarrow  \{l,r\}$. Let's record explicitly where are the occurrences of $l$ and $r$ in $\chi$. $\chi^{-1}(l)=\{i_1<i_2<\cdots <i_p\}$ and $\chi^{-1}(r)=\{i_{p+1}>i_{p+2}>\cdots >i_n\}$. Then we define a permutation $s_\chi$ on $\{1, 2, \cdots, n\}: s_\chi(k)=i_k$, for $k=1, 2,\cdots, n$.
 For a subset $V=\{i_1< i_2< \cdots< i_k\}$ of the set $\{1, 2, \cdots, n\}$, $a_1\cdots, a_n\in \A$, define $$\varphi_V(a_1, \cdots, a_n)=\varphi(a_{i_1}a_{i_2}\cdots a_{i_k}).$$  Let $\P(n)$ be the set of all partitions of $\{1, 2, \cdots, n\}$. For a partition $\pi=\{V_1, V_2, \cdots, V_d\}\in \P(n)$, we define $$\varphi_\pi(a_1, \cdots, a_n):=\prod_{V\in \pi}\varphi_V(a_1, \cdots, a_n).$$  Define $BNC(n,\chi)=\{s_\chi\circ \pi:\pi\in NC(n)\}$, where $NC(n)$ is the set of all non-crossing partitions of $\{1, 2, \cdots, n\}$ (Lecture 9 in \cite{NS}).  A $\sigma\in BNC(n,\chi)$ is called a {\sl bi-non-crossing} partition of $\{1, 2, \cdots, n\}$.  Let $(\A,\varphi)$ be a non-commutative probability space. The bi-free cumulants $(\kappa_\chi:\A^n\rightarrow \mathbb{C})_{n\ge 1, \chi:\{1, 2, \cdots, n\}\rightarrow \{l,r\}}$ of $(\A,\varphi)$ are defined by
 $$\kappa_\chi (a_1, \cdots, a_n)=\sum_{\pi\in BNC(n,\chi)}\varphi_{\pi}(a_1, \cdots, a_n)\mu_n(s^{-1}_\chi\circ\pi, 1_n),\eqno (1.1)$$ for $n\ge 1, \chi:\{1,2,\cdots, n\}\rightarrow \{l,r\}, a_1, \cdots, a_n\in A$, where $\mu_n$ is the Mobius function on $NC(n)$ (Lecture 10 in \cite{NS}). For a subset $V=\{i_1, i_2, \cdots, i_k\}\subseteq \{1, 2, \cdots, n\}$, let $\chi_V$ be the restriction of $\chi$ on $V$. We define $\kappa_{\chi, V}(a_1, a_2, \cdots, a_n)=\kappa_{\chi_V}(a_{i_1}, a_{i_2}, \cdots, a_{i_k})$.
 For a partition $\pi=\{V_1, V_2, \cdots, V_k\} \in BNC(n,\chi)$, we define $\kappa_{\chi, \pi}(a_1, \cdots, a_n)=\prod_{V\in \pi}\kappa_{\chi, V}(a_1, a_2, \cdots, a_n)$.  Then the bi-free cumulant appeared in $(1.1)$ is $\kappa_{\chi, 1_n}(a_1, \cdots, a_n)$. The bi-free cumulants are determined by the equation $$\varphi (a_1 a_2 \cdots a_n)=\sum_{\pi\in BNC(n, \chi)}\kappa_{\chi,\pi}(a_1, a_2, \cdots, a_n), \forall a_1, \cdots, a_n\in \A,   \eqno (1.2)$$ for a $\chi:\{1, 2, \cdots, n\}\rightarrow \{l,r\}$.

 Charlesworth, Nelson, and Skoufranis \cite{CNS2} proved that
  $$z'=((z'_i)_{i\in I}, (z'_j)_{j\in J}), z''=((z''_i)_{i\in I}, (z''_j)_{j\in J})$$ in a non-commutative probability space $(\A,\varphi)$ are bi-free if and only if
 $$\kappa_\chi(z_{\alpha(1)}^{\epsilon_1},z_{\alpha(2)}^{\epsilon_2}, \cdots, z_{\alpha(n)}^{\epsilon_n})=0, \eqno (1.3)$$ whenever $\alpha:\{1,2,\cdots, n\}\rightarrow I\bigsqcup J$, $\chi:\{1, 2, \cdots, n\}\rightarrow \{l,r\}$ such that $\alpha^{-1}(I)=\chi ^{-1}(\{l\})$,   $\epsilon:\{1,2, \cdots, n\}\rightarrow \{',''\}$ is not constant, and $n\ge 2$ (Theorem 4.3.1 in \cite{CNS2}).
\begin{Definition}[\cite{DV}, Section 7] Let $(\A,\varphi)$ be a non-commutative probability space and let $I$ and $J$ be tow disjoint index sets. A two-faced family $(\{s_i\}_{i\in I}, \{s_j\}_{j\in J})$ of random variables in $\A$ is said to have a (centered) bi-free central limit distribution if all bi-free cummulants of order 1 and of order at least 3 of the family are zeros. Specially, we call such a pair $(s_l, s_r)$ a bi-free pair of semicircular random variables.
\end{Definition}
\subsection{Structures of Operator-valued Bi-freeness}
Let $B$ be a unital algebra.
 A $B$-$B$-{\sl non-commutative probability space} is a triple $(\A,E,\varepsilon)$ where $\A$ ia a unital algebra, $\varepsilon: B\otimes B^{op}\rightarrow \A$ is a unital homomorphism such that $\varepsilon|_{B\otimes 1_B}$ and $\varepsilon|_{1_B\otimes B}$ are injective, and $E:\A\rightarrow B$ is a linear map such that $E(\varepsilon(b_1\otimes b_2)a)=b_1E(a)b_2$ and $E(a\varepsilon(b\otimes 1_B))=E(a\varepsilon(1_B\otimes b))$. Let $L_b=\varepsilon(b\otimes 1_B)$ and $R_b=\varepsilon(1_B\otimes b)$. The unital subalgebras $$\A_l=\{a\in \A: aR_b=R_ba, \forall b\in B\}, \A_r=\{a\in \A: L_ba=aL_b,\forall b\in B\}$$ are called the {\sl left } and {\sl right} algebras of $\A$, respectively.
The following we give a canonical example of $B$-$B$-non-commutative probability spaces.

A {\sl $B$-$B$-bi-module with a specified $B$-vector state} is a triple $(\mathcal{X}, \mathcal{X}^0, p)$ where $\mathcal{X}=B\oplus \mathcal{X}^0$, a direct sum of $B$-$B$ bi-modules and $p:\mathcal{X}\rightarrow B$, $p(b\oplus\eta)=b$. Let $\L(\X)$ denote the set of linear operators on $\mathcal{X}$. For $b\in B$, define $L_b, R_b\in \L(\X)$ by $$L_b(x)=bx, R_b(x)=xb, \forall x\in \X.$$ Similarly, we can define left and right algebras as follows $$\L_l(\X):=\{A\in \L(\X): AR_b=R_bA, \forall b\in B\}, \L_r(\X):=\{A\in \L(\X):AL_b=L_bA, \forall b\in B\}. $$ Given a $B$-$B$-bi-module with a specified $B$-vector state $\{\X, \X^0, p\}$, the expectation $E_{\L(\X)}$ of $\L(\X)$ onto $B$ is defined by $$E_{\L(\X)}(A)=p(A1_B), \forall A\in \L(\X).$$ Define $\varepsilon: B\otimes B^{op}\rightarrow \A$, $\varepsilon(b_1\otimes b_2)=L_{b_1}R_{b_2}$. Then $(\L(\X), E_{\L(\X)}, \varepsilon)$ is a (concrete) $B$-$B$-non-commutative probability space. Moreover, Theorem 3.2.4 in \cite{CNS1} demonstrated that every abstract $B$-$B$-non-commutative probability space can be represented inside a concrete $B$-$B$-non-commutative probability space.

\subsection{Bi-Matrix Models}
Let $\{E_{i,j}(N)\}_{i,j=1}^N$ denote the matrix unit system of $M_N(\mathbb{C})$, $A=[a_{i,j}]=\sum_{i,j=1}^Na_{i,j}\otimes E_{i,j}(N)$ denote a matrix in $M_N(\mathbb{C})$. Tr will denote the trace on $M_N(\mathbb{C})$ defined by $Tr(A)=\sum_{i=1}^Na_{i,i}$.
In the following, we introduce the general construction for bi-matrix models for matrices with elements in $\L(\X)$.

Let $(\X, \X^0, \xi, p)$ be a pointed vector space(see the beginning of Section 1.1), $p: \X\rightarrow \mathbb{C}$, $p(\lambda\xi\oplus \eta)=\lambda$. For $N\in \mathbb{N}$, consider $M_N(\mathbb{C})$-$M_N(\mathbb{C})$ bi-modular actions on $\X_N:=M_N(\X)$, $$[a_{i,j}]\cdot[\eta_{i,j}]=[\sum_{k=1}^Na_{i,k}\eta_{k,j}], [\eta_{i,j}]\cdot[a_{i,j}]=[\sum_{k=1}^Na_{k,j}\eta_{i,k}],$$ for all $[a_{i,j}]\in M_N(\mathbb{C})$ and $[\eta_{i,j}]\in \X_N$. Then $\X_N$ becomes an $M_N(\mathbb{C})$-$M_N(\mathbb{C})$-bi-module with a specified $M_N(\mathbb{C})$-vector sate via $$\X_N=M_N(\mathbb{C}\xi)\oplus M_N(\X^0),$$
and a linear map $p_{\X_N}:\X_N\rightarrow M_N(\mathbb{C})$ defined by $p_{\X_N}([\eta_{i,j}])=[p(\eta_{i,j})]$.

$\X_N$ is called the $M_N(\mathbb{C})$-$M_N(\mathbb{C})$ -bi-module associated with $(\X,p)$ and $(\L(\X_N), E_{\L(\X_N)}, \varepsilon)$ is called the $M_N(\mathbb{C})$-$M_N(\mathbb{C})$-non-commutative probability space associated with $(\X,p)$. The expectation $E_{\L(\X_N)}:\L(\X_N)\rightarrow M_N(\mathbb{C})$ has the form $E_{\L(\X_N)}(A)=p_{\X_N}(A1_{N,\xi})$, where $A\in \L(\X_N)$ and $1_{M,\xi}$ is the diagonal matrix $diag(\xi, \xi, \cdots, \xi)$.

To consider bi-matrix models, we define two homomorphisms $L: M_N(\L(\X))\rightarrow \L(\X_N)$, and $R: M_N(\L(\X)^{op})^{op}\rightarrow \L(\X_N)$, $$L([T_{i,j}])[\eta_{i,j}]=[\sum_{k=1}^NT_{i,k}(\eta_{k,j})], R([T_{i,j}])[\eta_{i,j}]=[\sum_{k=1}^NT_{k,j}\eta_{i,k}],$$ for $[\eta_{i,j}]\in \X_N$ and $[T_{i,j}]\in M_N(\L(\X))$.


\section{Main Results}

In this section, we shall construct  bi-matrix models of non-Gaussian random variables to approximate bi-free central limit distribution with a positive definite covariance matrix. We also prove asymptotic freeness of a family of   two-faced families of non-Gaussian random matrices from a two-faced family of constant diagonal matrices. We assume that all linear functionals on non-commutative probability spaces under consideration  are tracial.

Let $\A:=L^{\infty-}(\Omega, \mu)=\cap_{p\ge 1}L^p(\Omega, \mu)$, where $(\Omega, \mu)$ is a probability space, and  $E(f)=\int_\Omega f(t)d\mu$, for $f\in \A$. It was proved that $(\A,E)$ is a non-commutative probability (See Exercise 1.22 in \cite{NS}). Let $I$ and $J$ be two disjoint index sets. For $N\in \mathbb{N}$, an $N\times N$ {\sl random two-faced family of matrices} on $\A$ is a two-faced family  $((X^i)_{i\in I},(X^j)_{j\in J})$ where $X^i=L([X_{l,k;i}]_{N\times N})$ and $X^j=R([X_{l,k;j}]_{N\times N})$ are left and right matrices, respectively, with entries from $\A\subset \L(\A)$.

Let $I$ and $J$ be disjoint index sets, and $((s_i)_{i\in I}, (s_j)_{j\in J})$ be a two-faced family of self-adjoint random variables in a $*$-probability space, which  has a bi-free central limit distribution with covariance matrix $C$. Then $C$ is positive definite. In fact, for $m\in \mathbb{N}$ and $\chi:\{1, 2, \cdots, m\}\rightarrow I\cup J$, we have $$C_\chi=(c_{\chi(k), \chi(l)})_{m\times m}=(\varphi(s_{\chi(k)}s_{\chi(l)}))_{m\times m}=\varphi\left(\left(\begin{array}{c}s_{\chi(1)}\\
s_{\chi(2)}\\
\vdots\\
s_{\chi(m)}
\end{array}\right)\left(\begin{array}{cccc}s_{\chi(1)}&s_{\chi(2)}&\cdots&s_{\chi(m)}\end{array}\right)\right)\ge 0.$$ By Theorem 7.6 in \cite{DV}, the family has a bi-free central limit $*$-distribution. Moreover, $c_{k,l}=\varphi(s_ks_l)=\varphi(s_ls_k)=\overline{\varphi(s_ks_l)}\in \mathbb{R}$. It follows that $C$ is a positive definite matrix with real entries. Generally, having a bi-free central limit distribution does not necessarily imply that the covariance matrix is positive definite.  For instance, in the Fock space representation of such a distribution, $s_i=l(h(i))+l^*(h^*(i))$, for $i\in I$, $s_j=r(h(j))+r^*(h^*(j))$, for $j\in J$, where $h, h^*: I\cup J\rightarrow\H$ (Theorem 7.4 in \cite{DV}). The equality $c_{i,j}=\langle h(j), h^*(i)\rangle=\overline{c_{j,i}}$ holds if $h=h^*$.

The following definition modifies Definition 4.5 in \cite{PS}, adopting some properties of random matrices described in \cite{KD}. By the way, the entries of matrices in Definitions 4.1 and 4.5 in \cite{PS} should be taken from $L^{\infty-}$, not from $L^\infty(\Omega, \mu)$, because Gaussian random variables are unbounded.

\begin{Definition}
Let $I$ and $J$ be disjoint index sets, and $C=(c_{k, m})_{k,m\in I\coprod J}$ be a positive definite matrix in $M_{I\cup J}(\mathbb{C})$ with real entries. Let $(\Omega, P)$ be a probability space and $N\in \mathbb{N}$.   A self-adjoint $C$-random two-faced familiy of $N\times N$ matrices is an $N\times N$ two-faced family of matrices $((X^i_N)_{i\in I}, (X^j_N)_{j\in J})$ on $L^{\infty-}(\Omega, \mu)$ with $X^k_N=L([X^k_{i,j;N}]_{N\times N})$ for $k\in I$ and $X^k_N=R([X^k_{i,j;N}]_{N\times N})$ for $k\in J$, where
\begin{enumerate}
\item $E(X^k_{i,j;N})=0$, for all $i, j=1, 2, \cdots, N$, and $k\in I\cup J$,
\item ${X^k_{j,i;N}}=\overline{X^k_{i,j;N}}$, for all $i, j=1, 2, \cdots, N$, and $k\in I\cup J$,
\item $\{\{X^k_{i,j;N}: k\in J \cup J\}: 1\le i\le j\le N\}$ is an independent family of sets of random variables,
\item $E(X^{k_1}_{i,j;N}X^{k_2}_{j,i;N})=\frac{1}{N}c_{k_1, k_2}$, for all $i, j=1, 2, \cdots, N$, and $k_1, k_2\in I\cup J$,
\item Moreover, for an $m\in \mathbb{N}$ and a function $\chi: \{1, 2, \cdots, m\}\rightarrow I\cup J$, there exist a positive constant number $M(\chi, m)$ and $N_0\in \mathbb{N}$ such that $$\sup_{1\le i\le j\le N}|E(X^{\chi(1)}_{\lambda(1),\tau(1);N}X^{\chi(2)}_{\lambda(2), \tau(2);N}\cdots X_{\lambda(m),\tau(m);N}^{\chi(m)})|\le M(\chi, m)N^{-m/2},$$ whenever  $N>N_0$, where $(\lambda(l),\tau(l))=(i, j)$, or $(\lambda(l),\tau(l))=(j, i)$, for $l=1, 2, \cdots, m$.
\end{enumerate}
\end{Definition}

\begin{Theorem} Let $I$ and $J$ be two disjoint index sets, $C$ be a positive definite matrix of size $I\cup J$ with real entries, and  $((a_i)_{i\in I}, (a_j)_{j\in J})$ be a two-faced family of self-adjoint random variables having a bi-free central limit distribution in a $*$-probability space $(\A,\varphi)$ with covariance matrix $C$. For each $N\in \mathbb{N}$, let $((X^i_N)_{i\in I}, (X^j_N)_{j\in J})$ be a two-faced family of self-adjoint $C$-random two-faced family of $N\times N$ matrices defined in Definition 2.1.
Then for every $n\in \mathbb{N}$ and every function $\chi:\{1, 2, \cdots, n\}\rightarrow I\cup J$, we have
$$\lim_{N\rightarrow \infty}\frac{1}{N}Tr(E_{\L(\mathcal{X}_N)}(X_N^{\chi(1)} X_N^{\chi(2)}\cdots X^{\chi(n)}_N))=\varphi(a_{\chi(1)}a_{\chi(2)}\cdots a_{\chi(n)}).\eqno (2.1) $$
\end{Theorem}
\begin{proof}
Let $s_\chi$ be the permutation in $S_n$ related to $\chi$ defined in Section 1.  By Lemma 3.7 and Remark 3.8 in \cite{PS}, a summand  $E(X^{\chi(1)}_{i_1, j_1;N}\cdots X^{\chi(n)}_{i_n, j_n;N})$ of a diagonal entry of the matrix $E_{\L(\mathcal{X}_N)}(X^{\chi(1)}_N\cdots X^{\chi(n)}_N)$
must satisfy $j_{s_\chi(k)}=i_{s_\chi(k+1)}$, for $1\le k\le n-1$, and $j_{s_\chi(n)}=i_{s_\chi(1)}$. It follows that $j_k=i_{s_\chi(s_\chi^{-1}(k)+1)}$, for $k=1, 2, \cdots, n$, where $n+1$ is substituted by $1$. In this case, $$E(X^{\chi(1)}_{i_1, j_1;N}\cdots X^{\chi(n)}_{i_n, j_n;N})=E(X^{\chi(1)}_{i_1,i_{\eta(1)};N}X^{\chi(2)}_{i_2,i_{\eta(2)};N}\cdots X^{\chi(n)}_{i_n,i_{\eta(n)};N})$$  is a summand of the $(i_{s_\chi(1)}, i_{s_\chi(1)})$-entry of matrix $E_{\L(\mathcal{X}_N)}(X^{\chi(1)}_N\cdots X^{\chi(n)}_N)$, where $\eta(k)=s_\chi(s_\chi^{-1}(k)+1)$, $1\le k\le n$. Furthermore, the $(x,x)$-entry of the matrix is $$\alpha(x,x;N):=\sum_{1\le k\le  n, k\ne s_\chi(1), s_\chi(1)=x}\sum_{i_k=1}^N E(X^{\chi(1)}_{i_1,i_{\eta(1)};N}X^{\chi(2)}_{i_2,i_{\eta(2)};N}\cdots X^{\chi(n)}_{i_n,i_{\eta(n)};N}).$$
Therefore, the left hand side of (2.1) has the following form
$$\lim_{N\rightarrow \infty}\frac{1}{N}Tr(E_{L(\mathcal{X}_N)}(X_N^{\chi(1)} X_N^{\chi(2)}\cdots X^{\chi(n)}_N))=\lim_{N\rightarrow \infty}\frac{1}{N}\sum_{i_{s_\chi(1)}=1}^N\alpha(i_{s_\chi(1)}, i_{s_\chi(1)};N). $$
We want to find out summands of diagonal entries of the matrix  $E_{L(\mathcal{X}_N)}(X^{\chi(1)}_N\cdots X^{\chi(n)}_N)$ which have  possible non-zero contributions to the left hand side of (2.1). In order to evaluate an expression  $$E(X^{\chi(1)}_{i_1,i_{\eta(1)};N}X^{\chi(2)}_{i_2,i_{\eta(2)};N}\cdots X^{\chi(n)}_{i_n,i_{\eta(n)};N}),$$ for  given indices $i_1, \cdots, i_n$, we use the independence condition (3) in Definition 2.1, and collect $X^{\chi(k)}_{i_k, i_{\eta(k)};N}$'s together into mutually independent groups. Then we multiply together the expectations of  the groups.     Given a family $\{i_1, \cdots, i_n\}$ of numbers in $\{1, 2, \cdots, N\}$,  we can define a partition  $\pi$ of $\{1, 2, \cdots, n\}$, $\pi=\{V_1, V_2, \cdots, V_k\}$, where $x, y\in V$ if and only if $(I):$ $i_x=i_y$ and $i_{\eta(x)}=i_{\eta(y)}$, or $(II):$ $i_x=i_{\eta(y)}$ and $i_y=i_{\eta(x)}$. We can define an orientation function $\mathcal{O}:\{(x,y): x, y\in V, V\in \pi\}\rightarrow \{1, -1\}$ by $\mathcal{O}(x,y)=\left\{\begin{array}{ll}1,& \text{if } (I),\\
-1,& \text{if } (II)\end{array}\right.$. It follows from Property (3) in Definition 2.1 that $$E(X^{\chi(1)}_{i_1,i_{\eta(1)};N}X^{\chi(2)}_{i_2,i_{\eta(2)};N}\cdots X^{\chi(n)}_{i_n,i_{\eta(n)};N})=\prod_{V\in \pi}E(\prod_{x\in V}X^{\chi(x)}_{i_x, i_{\eta(x)}}).\eqno (2.2)$$ By Property (1) in Definition 2.1, the equation $(2.2)=0$, if there is a block $V\in \pi$ containing only one number. Therefore, we assume $|V|\ge 2$, for all $V\in \pi$. Moreover, by Property (5) in Definition 2.1, we have
$$|E(X^{\chi(1)}_{i_1,i_{\eta(1)};N}X^{\chi(2)}_{i_2,i_{\eta(2)};N}\cdots X^{\chi(n)}_{i_n,i_{\eta(n)};N})|\le\prod_{V\in \pi}|E(\prod_{x\in V}X^{\chi(x)}_{i_x, i_{\eta(x)}})|=\prod_{V\in \pi}O(N^{-|V|/2})=O(N^{-n/2}),$$ as $N\rightarrow \infty$. We use the following notations adopted from \cite{KD}.
\begin{enumerate}
\item A partition $\pi\in \P(n)$ has property $P$, if $|V|\ge 2$, for all $V\in \pi$.
\item A partition $\pi\in \P(n)$ has property $P_3$ if $\pi$ has property $P$ and there is a block $V\in \pi$ such that $|V|\ge3$.
\item $P_2(n)=\{\pi\in P(n): |V|=2, \forall V\in \pi\}$.
\item A partition $\pi\in \P(n)$ has property $P_4$ if $\pi\in P_2(n)$, and there is a  block $\{x,y\}\in \pi$ such that $\mathcal{O}(x,y)=1$.
\item A partition $\pi\in \P(n)$ has property $P_5$ if $\pi\in P_2(n)$, and $\forall \{x,y\}\in \pi, \mathcal{O}(x,y)=-1$.
\end{enumerate}
An $(x,x)$-entry $\alpha(x,x;N)$ has the following estimate
\begin{align*}
|\alpha(x,x;N)|\le &\sum_{\pi\in \mathcal{P}(n): \pi \text{ has property }P}(\sum_{i_k=1,k=1, 2, \cdots, n, k\ne s_\chi(1), \text{giving } \pi}^N O(N^{-n/2}))\\
=&\sum_{\pi\in \mathcal{P}(n): \pi \text{ has property } P}O(N^{-n/2})\theta_{xx}(\pi),
\end{align*}
as $N\rightarrow \infty$, where $\theta_{xx}(\pi)$ is the number of choices of $i_1, i_2, \cdots, i_n$ except $i_{s_\chi(1)}$ which give $\pi$.
If a partition $\pi\in \P(n)$ has properties $P_3$ or $P_4$,  then by the argument in the proof of Lemma 2.2 in \cite{KD} (see Pages38-40 of the paper), $\theta_{xx}(\pi)\le N^{\frac{n-1}{2}}$. Note that this upper bound of $\theta_{xx}(\pi)$  is independent of $x$. It follows that for such a $\pi\in \mathcal{P}(n)$, $O(N^{-n/2})\theta_{xx}(\pi)\le GN^{-n/2}N^{(n-1)/2}\rightarrow 0$, as $N\rightarrow \infty$, where $G>0$ is a constant (independent of $x$ and $N$). Therefore,
\begin{align*}
&\lim_{N\rightarrow \infty}\sup_{1\le i_{s_\chi(1)}\le N}|\sum_{\pi\in \P(n) \text{ having property } P_3 \text{ or } P_4}\\
&(\sum_{i_k=1,k=1, 2, \cdots, n, k\ne s_\chi(1), \text{giving } \pi}^N E(X^{\chi(1)}_{i_1,i_{\eta(1)};N}X^{\chi(2)}_{i_2,i_{\eta(2)};N}\cdots X^{\chi(n)}_{i_n,i_{\eta(n)};N})|\\
\le & \lim_{N\rightarrow \infty}\sup_{1\le i_{s_\chi(1)}\le N}\sum_{\pi\in \P(n) \text{ having property } P_3 \text{ or } P_4}(\sum_{i_k=1,k=1, 2, \cdots, n, k\ne s_\chi(1), \text{giving } \pi}^N O(N^{-n/2})\\
= &\lim_{N\rightarrow \infty}\sup_{1\le i_{s_\chi(1)}\le N}\sum_{\pi\in \P(n) \text{ having property } P_3 \text{ or } P_4}O(N^{-n/2})\theta_{xx}(\pi)\\
\le & \lim_{N\rightarrow \infty}\sup_{1\le i_{s_\chi(1)}\le N}\sum_{\pi\in \P(n) \text{ having property } P_3 \text{ or } P_4} G N^{-1/2}\\
=&\lim_{N\rightarrow \infty}\sup_{1\le i_{s_\chi(1)}\le N}N^{-1/2}\varepsilon(n)=0,
\end{align*}
where $\varepsilon(n)$ is the number of partitions in $\P(n)$ which have properties $P_3$ or $P_4$.
 It implies that
\begin{align*}
&\alpha(x,x;N)\\
=&\sum_{\pi \in P_2(n) \text{ having Property } P_5}(\sum_{i_k, k=1,2,\cdots, n, k\ne s_\chi(1) \text{giving } \pi}E(X^{\chi(1)}_{i_1,i_{\eta(1)};N}X^{\chi(2)}_{i_2,i_{\eta(2)};N}\cdots X^{\chi(n)}_{i_n,i_{\eta(n)};N}))+O(\varepsilon)\\
=&\sum_{\pi \in P_2(n)\text{ having Property } P_5}(\sum_{i_k, k=1,2,\cdots, n, k\ne s_\chi(1) \text{giving } \pi}\prod_{\{x,y\}\in \pi}E(X^{\chi(x)}_{i_x,i_{\eta(x)}}X^{\chi(y)}_{i_y,i_{\eta(y)}}))+O(\varepsilon)\\
=&\sum_{i_k=1,  k=1,2,\cdots, n, k\ne s_\chi(1)}^N\sum_{\pi\in \P_2(n)}\frac{1}{N^{n/2}}\prod_{\{x,y\}\in\pi}\delta_{i_x, i_{\eta(y)}}\delta_{i_y, i_{\eta(x)}}c_{\chi(x), \chi(y)}+O(\varepsilon),
\end{align*}
where $\lim_{N\rightarrow\infty}O(\varepsilon) =0$.
Therefore, we have
\begin{align*}
\text{LHS of } (2.1)&=\lim_{N\rightarrow\infty}\frac{1}{N}\sum_{i_1, i_2, \cdots, i_n=1}^N\sum_{\pi\in\P_2(n)}\frac{1}{N^{n/2}}\prod_{\{x,y\}\in \pi}\delta_{i_x, i_{\eta(y)}}\delta_{i_y, i_{\eta(x)}}c_{\chi(x), \chi(y)}\\
=&\lim_{N\rightarrow\infty}\frac{1}{N}\sum_{j_1, j_2, \cdots, j_n=1}^N\sum_{\pi\in\P_2(n)}\frac{1}{N^{n/2}}\prod_{\{x,y\}\in \pi}\delta_{j_{s_\chi^{-1}(x)}, j_{s_\chi^{-1}(y)+1}}\delta_{j_{s_\chi^{-1}(y)}, j_{s_\chi^{-1}(x)+1}}c_{\chi(x), \chi(y)},
\end{align*}
where the last equality follows by replacing $i_x$ by $j_{s_\chi^{-1}(x)}$ (therefore, $i_{\eta(y)}=j_{s_\chi^{-1}(y)+1}$). We have got the same equation as that in the meddle in Page 10 of \cite{PS}. The rest of the present proof is same as the corresponding part of Theorem 4.10 in \cite{PS}.
\end{proof}

\begin{Remark} When $\chi:\{1, 2, \cdots, n\}\rightarrow I$, the above theorem gives a non-Gaussian random matrix model for a semicircular family of (possibly singular) covariance $C_I$( the restriction of the above positive definite matrix $C$ to $I\times I$) defined in Definition 8.15 in \cite{NS}. While Remark 4.12 in \cite{PS} provides a Gaussian random matrix model for such a semicircular family with a non-singular covariance matrix.
\end{Remark}
Theorem 4.11 in \cite{PS} gives a result on asymptotic bi-freeness of independent random pairs of matrices with Gaussian entries.  We can get a similar result for random  two-faced  families  of matrices with non-Gaussian entries.
\begin{Theorem}
Let $K$ be an index set. Let $C_k=(c_{i,j;k})_{I\cup J\times I\cup J}$ be a positive definite matrix with real entries for $k\in K$, and let $\{((s_{i,k})_{i\in I}, (s_{j,k})_{j\in J}):k\in K\}$ be a bi-free collection of bi-free central limit distributions of self-adjoint random variables in a  $*$-probability space $(\A,\varphi)$ with covariance matrix $C_k$, for $k\in K$. For each $N\in \mathbb{N}$ and $k\in K$, let $((X^{i,k}_N)_{i\in I}, (X^{j,k}_N)_{j\in J})$ be a self-adjoint $C_k$-random two-faced family of $N\times N$ matrices. Moreover, suppose that
 $$ \{\{X^{l,k}_{i,j;N}:l\in I\cup J, k\in K\}: 1\le i\le j\le N\}\eqno (2.3)$$ is an independent family of  sets of random variables in $L^{\infty-}(\Omega,\mu)$, for an $m\in \mathbb{N}$ and  functions $$\chi: \{1, 2, \cdots, m\}\rightarrow I\cup J, \ \epsilon:\{1, 2, \cdots, m\}\rightarrow K,$$ there exist a positive constant number $M(\chi, \epsilon, m)$ and $N_0\in \mathbb{N}$ such that $$\sup_{1\le i\le j\le N}|E(X^{\chi(1),\epsilon(1)}_{\lambda(1),\tau(1);N}X^{\chi(2), \epsilon(2)}_{\lambda(2),\tau(2) ;N}\cdots X_{\lambda(m),\tau(m);N}^{\chi(m), \epsilon(m)})|\le M(\chi, \epsilon,m)N^{-m/2},\eqno (2.4)$$ whenever  $N>N_0$, where $(\lambda(l),\tau(l))=(i,j)$, or $(\lambda(l), \tau(l))=(j,i)$, for $l=1, 2, \cdots, m$,  and $$E(X_{i,j;K}^{l_1, k_1}X_{j,i;N}^{l_2, k_2})=0, \forall k_1\ne k_2, i\in I, i\in I, j\in J, k_1, k_2 \in K.\eqno (2.5)$$
 Then, for every $n\in \mathbb{N}$ and functions $\chi: \{1, 2, \cdots, n\}\rightarrow I\cup J, \epsilon:\{1,2, \cdots, n\}\rightarrow K$, we have
$$\lim_{N\rightarrow \infty}\frac{1}{N}Tr(E_{\L(\mathcal{X}_N)}(X_N^{\chi(1),\epsilon(1)} X_N^{\chi(2), \epsilon(2)}\cdots X^{\chi(n), \epsilon(n)}_N))=\varphi(s_{\chi(1), \epsilon(1)}s_{\chi(2),\epsilon(2)}\cdots s_{\chi(n),\epsilon(n)}).\eqno(2.6)$$
\end{Theorem}
\begin{proof}
We use ideas and notations in the proof of Theorem 2.2. But in this case, we need consider one more factor, the index function $\epsilon: \{1, 2, \cdots, n\}\rightarrow K$.  We will use the properties (1)-(5) in Definition 2.1 and (2.3)-(2.5) in this theorem to estimate  $$E(X^{\chi(1), \epsilon(1)}_{i_1,i_{\eta(1)};N}X^{\chi(2), \epsilon(2)}_{i_2,i_{\eta(2)};N}\cdots X^{\chi(n), \epsilon(n)}_{i_n,i_{\eta(n)};N}),$$ a summand of the $(i_{s_\chi(1)}, i_{s_\chi(1)})$-entry of matrix $E_{\L(\mathcal{X}_N)}(X^{\chi(1), \epsilon(1)}_N\cdots X^{\chi(n), \epsilon(n)}_N)$.

 For $\{i_1, i_2, \cdots, i_n\}\subseteq \{1, 2, \cdots, N\}$, we define a partition $\pi$ according to the conditions $(I)$ and $(II)$ in the proof of Theorem 2.2. By Condition (2.3) of this theorem, we get an equation similar to $(2.2)$
$$E(X^{\chi(1),\epsilon(1)}_{i_1,i_{\eta(1)};N}X^{\chi(2),\epsilon(2)}_{i_2,i_{\eta(2)};N}\cdots X^{\chi(n),\epsilon(n)}_{i_n,i_{\eta(n)};N})=\prod_{V\in \pi}E(\prod_{x\in V}X^{\chi(x),\epsilon(x)}_{i_x, i_{\eta(x)}}).$$
Moreover, by  (2.4), we have
$$|E(X^{\chi(1),\epsilon(1)}_{i_1,i_{\eta(1)};N}X^{\chi(2), \epsilon(2)}_{i_2,i_{\eta(2)};N}\cdots X^{\chi(n), \epsilon(n)}_{i_n,i_{\eta(n)};N})|\le\prod_{V\in \pi}|E(\prod_{x\in V}X^{\chi(x), \epsilon(n)}_{i_x, i_{\eta(x)}})|=\prod_{V\in \pi}O(N^{-|V|/2})=O(N^{-n/2}),$$ as $N\rightarrow \infty$.
By $(2.5)$ and the same arguments as those after $(2.2)$, we get  \begin{align*}
&\text{LHS } of (2.6) \\
=&\lim_{N\rightarrow \infty}\frac{1}{N}\sum_{s_\chi(1)}^N\sum_{\pi \in P_2(n) \text{ having Property } P_5}(\sum_{i_k, k=1,2,\cdots, n, k\ne s_\chi(1) \text{ giving } \pi}E(X^{\chi(1),\epsilon(1)}_{i_1,i_{\eta(1)};N}X^{\chi(2), \epsilon(2)}_{i_2,i_{\eta(2)};N}\cdots X^{\chi(n),\epsilon(n)}_{i_n,i_{\eta(n)};N}))\\
=&\lim_{N\rightarrow \infty}\frac{1}{N}\sum_{s_\chi(1)}^N\sum_{\pi \in P_2(n)\text{ having Property } P_5}(\sum_{i_k, k=1,2,\cdots, n, k\ne s_\chi(1) \text{ giving } \pi}\prod_{\{x,y\}\in \pi}E(X^{\chi(x),\epsilon(x)}_{i_x,i_{\eta(x)}}X^{\chi(y), \epsilon(y)}_{i_y,i_{\eta(y)}})\\
=&\lim_{N\rightarrow \infty}\frac{1}{N}\sum_{s_\chi(1)}^N\sum_{\pi \in P_2(n)\text{ having Property } P_5}(\sum_{i_k, k=1,2,\cdots, n, k\ne s_\chi(1) \text{ giving } \pi}\prod_{\{x,y\}\in \pi}\frac{1}{N}c_{\chi(x), \chi(y);\epsilon(x)}\delta_{\epsilon(x), \epsilon(y)}).
\end{align*}
For the given function $\epsilon:\{1, 2, \cdots, n\}\rightarrow K$, define a partition $\pi_\epsilon\in \P(n)$ by the rule that two numbers $x,y$ are in the same block of $\pi_\epsilon$ if and only if $\epsilon(x)=\epsilon(y)$. For a partition $\pi\in \P(n)$, we say that $\pi\preceq \pi_\epsilon$ if every block of $\pi_\epsilon$ is the union of some blocks of $\pi$. We then have
$$\text{LHS } of (2.6)=\lim_{N\rightarrow\infty}\frac{1}{N}\sum_{i_1, i_2, \cdots, i_n=1}^N\sum_{\pi\in\P_2(n), \pi\preceq \pi_\epsilon}\frac{1}{N^{n/2}}\prod_{\{x,y\}\in \pi}\delta_{i_x, i_{\eta(y)}}\delta_{i_y, i_{\eta(x)}}c_{\chi(x), \chi(y);\epsilon(x)}.\eqno (2.7)$$ By the last part of the proof of Proposition 4.10 in \cite{PS} and $(2.7)$, we have
\begin{align*}
\text{LHS } of (2.6)=&\sum_{\pi\in BNC_2(n,\chi), \pi\preceq\pi_\epsilon}\prod_{\{x,y\}\in \pi}c_{\chi(x), \chi(y),\epsilon(x)}\\
=&\sum_{\pi\in BNC_2(n,\chi), \pi\preceq\pi_\epsilon}\kappa_{\pi,\chi}(s_{\chi(1), \epsilon(1)}, s_{\chi(2), \epsilon(2)}, \cdots, s_{\chi(n), \epsilon(n)})\\
=&\sum_{\pi\in BNC_2(n,\chi)}\kappa_{\pi,\chi}(s_{\chi(1), \epsilon(1)}, s_{\chi(2), \epsilon(2)}, \cdots, s_{\chi(n), \epsilon(n)})\\
=&\sum_{\pi\in BNC(n, \chi)}\kappa_{\pi,\chi}(s_{\chi(1), \epsilon(1)}, s_{\chi(2), \epsilon(2)}, \cdots, s_{\chi(n), \epsilon(n)})\\
=&\varphi(s_{\chi(1), \epsilon(1)}, s_{\chi(2), \epsilon(2)}, \cdots, s_{\chi(n), \epsilon(n)}),
\end{align*}
where the second equality follows by the definition of bi-free cumulants, the third equality holds because of the bi-freeness of  $((s_{i,k})_{i\in I}, (s_{j,k})_{j\in J})$, for different $k\in K$, and the fourth equality holds by the definition of bi-free central limit distributions.
\end{proof}
\begin{Remark}
\begin{enumerate}
\item If $\{\{X_{i,j;N}^{l, k}:l\in I\cup J, 1\le i\le j\le N\}:k\in K\}$ is an independent family of sets of random variables, then it is obvious that $(2.3)$ and $(2.5)$ hold. Moreover, \begin{align*}&\sup_{1\le i\le j\le N}|E(X^{\chi(1),\epsilon(1)}_{\lambda(1),\tau(1);N}X^{\chi(2), \epsilon(2)}_{\lambda(2),\tau(2) ;N}\cdots X_{\lambda(m),\tau(m);N}^{\chi(m), \epsilon(m)})|\\
\le &\sup_{1\le i\le j\le N}\prod_{V\in \pi_\epsilon}|E(X^{\chi(1),\epsilon(V)}_{\lambda(1),\tau(1);N}X^{\chi(2), \epsilon(V)}_{\lambda(2),\tau(2) ;N}\cdots X_{\lambda(m),\tau(m);N}^{\chi(m), \epsilon(V)})| \\
\le &(\prod_{V\in \pi_\epsilon}M(\chi, \epsilon(V),m))N^{-m/2},
\end{align*}
where $\epsilon(V)$ is the common value of $\epsilon$ when restricted to $V\in \pi_\epsilon$. It implies that the independence of $\{\{X_{i,j;N}^{l, k}:l\in I\cup J, 1\le i\le j\le N\}:k\in K\}$ implies the three conditions $(2.3), (2.4)$, and $(2.5)$.
\item Theorem 4.11 in \cite{PS} states that independent self-adjoint Gaussian random pairs of matrices are asymptotic bi-free with the limit of a bi-free family of bi-free central limit distributions. The above theorem gives the same asymptotic freeness result for  a family of random two-faced families of matrices  under  conditions $(2.3)$, $(2.4)$, and $(2.4)$, which are weaker than the independence of $\{\{X_{i,j;N}^{l, k}:l\in I\cup J, 1\le i\le j\le N\}:k\in K\}$  by (1) in this remark.
\end{enumerate}
\end{Remark}

\begin{Theorem}Fix an index set $K$. Let $C_k=(c_{i,j;k})_{I\cup J\times I\cup J}$ be a positive definite matrix with real entries for $k\in K$. For each $N\in \mathbb{N}$ and $k\in K$, let $((Y^{i,k}_N)_{i\in I}, (Y^{j,k}_N)_{j\in J})$ be  a family of $N\times N$ matrices which satisfy the conditions (1)-(5) in Definition 2.1 and  (2.3)-(2.5). Let
$D_{k',N}=\sum_{i=1}^Nd_{i, N}^{k'}E(i, i;N), k'\in K'$, be diagonal constant matrices with the joint distribution converging to the distribution of $d_{k'}: k'\in K'$, as $N\rightarrow \infty$, where $d_{k'}, k'\in K'$, are elements in a non-commutative probability space $(\A, \varphi)$. Assume that $$\sup_{1\le i\le N, N=1, 2, \cdots}\{|d_{i,N}^{k'}|\}=M(k')<\infty,\forall k'\in K'.\eqno (2.8)$$  Then  $\{Y_N^{i,k}, Y_N^{j,k}, D_{k',N}: i\in I, j\in J, k\in K, k'\in K'\}$ converges in distribution
to that of $\{s_{i,k}, s_{j,k}, d_{k'}:i\in I, j\in J, k\in K,  k'\in K'\}$, as $N\rightarrow \infty$,  where $\{\{(s_{i,k})_{i\in I}, (s_{j,k})_{j\in J}\},k\in K,  \{d_{k'}: k'\in K'\} \}$ is a free family.
\end{Theorem}
\begin{proof}
Without loss of generality, we assume that $ \{D_{k',K}:k'\in K\}$ forms a multiplicative semigroup, for $N\in \mathbb{N}$.   In order to prove the statement on convergence in distributions, it is sufficient to prove that for any $n\ge 2, n\in \mathbb{N}$, and functions
$$\chi:\{1, 2, \cdots, n\}\rightarrow I\cup J,\  \alpha:\{1, 2, \cdots, n\}\rightarrow K,\  \beta:\{1, 2, \cdots, n\}\rightarrow K',$$ we have
$$\lim_{N\rightarrow \infty}\varphi_N(Y^{\chi(1), \alpha(1)}_{N}D_{\beta(1), N}Y^{\chi(2), \alpha(2)}_{N}D_{\beta(2), N}\cdots Y^{\chi(n), \alpha(n)}_{N}D_{\beta(n), N})$$
$$=\varphi(s_{\chi(1), \alpha(1)}d_{\beta(1)}\cdots s_{\chi(n),\alpha(n)}d_{\beta(n)}), \eqno (2.9)$$
where  $\varphi_N(\cdot)=\frac{1}{N}Tr(E_{\L_{\mathcal{X}_N}}(\cdot))$.

Let $Y^{\chi(l), \alpha(l)}_{N}=[x_{i,j;N}^{\chi(l), \alpha(l)}]_{N\times N}$ and $D_{\beta(l), N}=[d_{i,N}^{\beta(l)}]_{N\times N}$, where $l=1, 2, \cdots, n$. Let $$ \gamma(i)=i=1, i=1, 2, \cdots, n-1, \gamma(n)=1.$$ We use $LHS$ to denote the expression  in the left-hand side of $(2.6)$ before taking limit, we then have
\begin{align*}
LHS
=&\frac{1}{N}\sum_{i(1), \cdots, i(n)
=1}^NE( \prod_{l=1}^nx_{i(l), i(\gamma(l)); N}^{\chi(l), \alpha(l)}d_{i(\gamma(l)), N}^{\beta(l)})\\
=&\frac{1}{N}\sum_{i(1), \cdots, i(n)=1}^NE(\prod_{l=1}^n x_{i(l), i(\gamma(l)); N}^{\chi(l), \alpha(l)})
\prod_{l=1}^nd_{i(\gamma(1)), N}^{\beta(l)}.
\end{align*}
We use the methods in the proofs of Theorem 2.2 (see also Lemma 2.2 in \cite{KD}) to estimate the non-zero contributions of entries $E(\prod_{l=1}^nx_{i(l), i(\gamma(l)); N}^{\chi(l), \alpha(l)})$ to   $\lim_{N\rightarrow \infty}LHS$. For given $i(l), l=1, 2, \cdots, n$, we define a partition $\pi$ of $\{1, 2, \cdots, n\}$ as follows. Two numbers $p$ and $q$ are in the same block of $\pi$ if and only if
\begin{enumerate}
\item $i(p)=i(q)$, $i(\gamma(p))=i(\gamma(q))$,
or,
\item $i(p)=i(\gamma(q))$, $i(\gamma(p))=i(q)$.
\end{enumerate}
We can define an orientation function $\mathcal{O}:\{(p,q): p,q\in V, V\in \pi \}\rightarrow \{-1, 1\}$, $\mathcal{O}(p,q)=1$ if (1) holds; $\mathcal{O}(p,q)=-1$, if (2) holds.
By (2.3), we have $$E(\prod_{l=1}^nx_{i(l), i(\gamma(l)); N}^{\chi(l), \alpha(l)})=\prod_{V\in \pi}E(\prod_{p\in V}x_{i(p), i(\gamma(p));N}^{\chi(p), \alpha(p)}).$$ If $\pi$ contains a single block (i.e., a block contains only one element), then by condition (1) in Definition 2.1, $E(\prod_{l=1}^nx_{i(l), i(\gamma(l)); N}^{\chi(l), \alpha(l)})=0$. Thus, we assume  $|V|\ge 2$, for every $V\in \pi$. By (2.4), for the given functions $\chi$ and $\alpha$, and $n\in \mathbb{N}$, there is an $N_0\in \mathbb{N}$ such that $$|E(\prod_{l=1}^nx_{i(l), i(\gamma(l)); N}^{\chi(l), \alpha(l)})|\le \prod_{V\in \pi}|E(\prod_{p\in V}x_{i(p), i(\gamma(p));N}^{\chi(p), \alpha(p)})|\le \prod_{V\in \pi}M(\chi,\alpha|_V,  n)N^{-n/2}, \forall N\ge N_0.$$ Combining this inequality with (2.8), we get
$$LHS\le \frac{1}{N}\sum_{\pi \text{ has Property } P}\sum_{1\le i(1), \cdots, i(n),  \text{giving } \pi}\prod_{V\in \pi}M(\chi,\alpha|_V, n)\prod_{i=1}^nM(\beta(i))N^{-n/2},$$
$\forall N\ge N_0. $
Let $\theta(\pi)$ be the number of all possible choices of $i(1), \cdots, i(n)$ from $\{1, 2, \cdots, N\}$, which give the partition $\pi$. We then have $$LHS\le \frac{1}{N} \sum_{\pi\text{ has Property } P}\prod_{V\in \pi}M(\chi,\alpha|_V, n)\prod_{i=1}^nM(\beta(i))N^{-n/2}\theta(\pi),\eqno (2.10)$$ for $N\ge N_0$.

Let $P=\{\pi: \pi \text{ is a partition of } \{1, 2, \cdots, n\}, |V|\ge 2, \forall V\in \pi\}$. As in the proof of Theorem 2.2, we can divide the Set $P$ into three parts, $P=P_3\cup P_4\cup P_5$ (see the proof of Theorem 2.2 for definitions of $P_3, P_4$, and $P_5$).  By the proof of Theorem 2.4, we have
$$\lim_{N\rightarrow \infty}\frac{1}{N}\sum_{\pi\in P_3\cup P_4}\sum_{1\le i(1), \cdots, i(n)\le N,  \text{ giving } \pi}|E(\prod_{l=1}^nx_{i(l), i(\gamma(l)); N}^{\chi(l), \alpha(l)}d_{i(\gamma(l)), N}^{\beta(l)})|=0.$$

For $\pi\in P_5$, $\pi=\{V_1, \cdots, V_{n/2}\}$, and $|V_i|=2$, for $p,q\in V_i$, we have $$i(p)=i(\gamma(q))=i(\gamma(\widetilde{\pi}(p))), i(\gamma(\widetilde{\pi}(q)))=i(q),$$ for $i=1, 2, \cdots, n/2$, where $\widetilde{\pi}\in S_n$ is the permutation associated with $\pi$ defined by $\widetilde{\pi}(p)=q, \widetilde{\pi}(q)=p$, for $\{p,q\}\in \pi$. We thus have $i(\gamma(\widetilde{\pi}(l)))=i(l)$, for all $l=1, 2, \cdots, n$. In this case, $n$ must be even. It follows that
\begin{align*}
&\lim_{N\rightarrow \infty}LHS\\
=&\lim_{N\rightarrow \infty}\frac{1}{N}\sum_{\pi\in P_5}\sum_{1\le i(1), \cdots, i(n)\le N, i(1), \cdots, i(n) \text{ give } \pi}\prod_{\{p,q\}\in \pi}E(x_{i(p), i(\gamma(p));N}^{\chi(p), \alpha(p)}x_{i(q), i(\gamma(q)); N}^{\chi(q), \alpha(q)})\prod_{l=1}^nd_{i(\gamma(1)), N}^{\beta(l)}\\
=&\lim_{N\rightarrow \infty}\frac{1}{N}\sum_{\pi\in P_5}\prod_{\{p,q\}\in \pi}c_{\chi(p), \chi(q);\alpha(p)}\delta_{\alpha(p), \alpha(q)}\\
\times &(\sum_{1\le i(1), \cdots, i(n)\le N}\prod_{\{p,q\}\in \pi}\delta_{i(p+1),i(q)}\delta_{i(q+1),i(p)} N^{-n/2}\prod_{l=1}^nd_{i(\gamma(l)), N}^{\beta(l)})\\
=&\lim_{N\rightarrow \infty}\frac{1}{N}\sum_{\pi\in P_5, \pi\preceq\pi_\alpha}(\prod_{\{p,q\}\in \pi}c_{\chi(p), \chi(q); \alpha(p)}\sum_{1\le i(1), \cdots, i(n)\le N}\prod_{l=1 }^n\delta_{i(\gamma(\widetilde{\pi}(l))),i(l)} N^{-n/2}\prod_{l=1}^nd_{i(\gamma(l)), N}^{\beta(l)}\\
\end{align*}
Let $P_2$ be the set of all pairing partitions of $\{1, 2, \cdots, n\}$. For every $\pi\in P_2$, let $$\pi=\{\{p_1, q_1\}, \cdots \{p_{n/2}, q_{n/2}\}\}.$$ Choose $1\le i(p_l)=i(\gamma(q_l)), i(\gamma(p_l))=i(q_l)\le N$, and $l=1, 2, \cdots, n/2$. Then $\pi\in P_5$. Let $P_2^\alpha=\{\pi: \pi\in P_2, \pi\preceq\pi_\alpha\}$.
We thus get
$$\lim_{N\rightarrow \infty}LHS=\lim_{N\rightarrow \infty}\frac{1}{N}\sum_{\pi\in P_2^\alpha}\prod_{\{p,q\}\in \pi}c_{\chi(p), \chi(q);\alpha(p)}(\sum_{1\le i(1), \cdots, i(n)\le N}\prod_{l=1 }^n\delta_{i(\gamma(\widetilde{\pi}(l))),i(l)} N^{-n/2}\prod_{l=1}^nd_{i(\gamma(l)), N}^{\beta(l)}).$$
Let $\sharp(\zeta)$ be the number of cycles in the permutation $\zeta=\gamma\widetilde{\pi}$. By Lemma 22.31 in \cite{NS},  $$\sum_{1\le i(1), \cdots, i(n)\le N}\prod_{l=1 }^n\delta_{i(\gamma(\widetilde{\pi}(l))),i(l)} \prod_{l=1}^nd_{i(\gamma(l)), N}^{\beta(l)}=N^{\sharp(\zeta)}tr_{\widetilde{\pi}\gamma}(D_{\beta(1), N}\cdots D_{\beta(n), N}),$$ where  $tr:M_N(\mathbb{C})\rightarrow \mathbb{C}$ is the normalized trace. It follows that
\begin{align*}\lim_{N\rightarrow \infty}LHS=&\lim_{N\rightarrow \infty}\sum_{\pi\in P_2^\alpha}(\prod_{\{p,q\}\in \pi}c_{\chi(p), \chi(q);\alpha(p)})tr_{\widetilde{\pi}\gamma}(D_{\beta(1), N}\cdots D_{\beta(n), N})N^{\sharp(\zeta)-1-n/2}\\
=&\sum_{\pi\in P_2^\alpha}(\prod_{\{p,q\}\in \pi}c_{\chi(p), \chi(q);\alpha(p)})\lim_{N\rightarrow\infty}tr_{\widetilde{\pi}\gamma}(D_{\beta(1), N}\cdots D_{\beta(n), N})\lim_{N\rightarrow \infty}N^{\sharp(\zeta)-1-n/2}\\
=&\sum_{\pi\in P_2^\alpha\cap NC(n)}(\prod_{\{p,q\}\in \pi}c^{\alpha(p)}_{\chi(p), \chi(q)})\varphi_{\widetilde{\pi}\gamma}(d_{\beta(1)}\cdots d_{\beta(n)}),
\end{align*}
where the last equality comes from the fact that for $\pi\in P_2(n)$, $\sharp(\gamma\pi)\le 1+\frac{n}{2}$, and $\sharp(\gamma\pi)=1+\frac{n}{2}$ if and only if $\pi$ is non-crossing (Exercise 22.15 in \cite{NS}). The other fact we used here is that $\{D_{k',N}: k'\in K'\}$ converges in distribution to $\{d_{k'}: k'\in K'\}$, as $N\rightarrow \infty$.

On the other hand, since $((s_{i, k})_{i\in I}, (s_{j,k})_{j\in J}), k\in K$, and  $\{d_{k'}:k'\in K'\}$ are free,  by the moment formula of products of free random variables (Theorem 14.4 in \cite{NS}), we have
\begin{align*}
&\varphi(s_{\chi(1), \alpha(1)}d_{\beta(1)}\cdots s_{\chi(n),\alpha(n)}d_{\beta(n)})\\
=&\sum_{\pi\in NC(n)}\kappa_\pi(s_{\chi(1), \alpha(1)}, \cdots, s_{\chi(n), \alpha(n)})\varphi_{K(\pi)}(d_{\beta(1)}, \cdots, b_{\beta(n)})\\
=&\sum_{\pi\in NC_2^\alpha(n)}\kappa_\pi(s_{\chi(1), \alpha(1)}, \cdots, s_{\chi(n), \alpha(n)})\varphi_{K(\pi)}(d_{\beta(1)}, \cdots, b_{\beta(n)})\\
=&\sum_{\pi\in P_2^\alpha\cap NC(n)}\prod_{\{p,q\}\in \pi}\varphi(s_{\chi(p), \alpha(p)}s_{\chi(q), \alpha(p)})\varphi_{K(\pi)}(d_{\beta(1)}, \cdots, b_{\beta(n)})\\
=&\sum_{\pi\in P_2^\alpha\cap NC(n)}\prod_{\{p,q\}\in \pi}c_{\chi(p), \chi(q); \alpha(p)}\varphi_{K(\pi)}(d_{\beta(1)}, \cdots, b_{\beta(n)})\\
=&\sum_{\pi\in P_2^\alpha\cap NC(n)}\prod_{\{p,q\}\in \pi}c_{\chi(p), \chi(q); \alpha(p)}\varphi_{\widetilde{\pi}\gamma}(d_{\beta(1)}, \cdots, b_{\beta(n)}),
\end{align*}
where $K(\pi)$ is the Kreweras complement partition of $\pi$ defined in Definition 9.21 in \cite{NS}. By the discussion in Page 376 in \cite{NS}, $K(\pi)=\pi\gamma(=\widetilde{\pi}\gamma)$, if $\pi$ is a non-crossing pairing partition. We have proved (2.9).
\end{proof}
\begin{Remark}
Dykema proved an asymptotic freeness result for family of non-Gaussian random matrices from constant block diagonal matrices in \cite{KD} (Theorem 2.1, \cite{KD} ) similar to Theorem 2.6. The family of random matrices in Theorem 2.6 has an arbitrary positive definite covariance matrix. While the family in \cite{KD} has covariance matrix $I$, the identity matrix. Also, we give a shorter proof than that in \cite{KD}, by applying  some techniques in Lecture 22 in \cite{NS} to  our case.
\end{Remark}
Theorem 4.13 in \cite{PS} shows that asymptotic bi-freeness is pretty much asymptotic freeness  provided all left operators commute with all right operators. The follow result, as a simple consequence of the above theorem and Theorem 4.13 in \cite{PS},  gives such an example that asymptotic freeness implies asymptotic bi-freeness.

\begin{Corollary}
Let $K$ and $K'$ be disjoint index sets, $C_k, ((Y^{i,k}_{N})_{i\in I}, (Y_{N}^{j,k})_{j\in J}), k\in K$ and $D_{k', N}, k'\in K'$ be matrices defined in Theorem 2.6. Then $((L(Y_{N}^{i,k}))_{i\in I}, (R(Y_{N}^{j,k}))_{j\in J}), (L(D_{k'}), R(D_{k'})), k\in K, k'\in K'$,   are  asymptotically bi-free.
\end{Corollary}
\begin{proof}
Let $X_N^{l, k}=L(Y_{N}^{l, k})$, $X_N^{k'}=L(D_{N, k'})$, $Z_N^{l, k}=R(Y_{N}^{l, k})$, $Z_N^{k'}=R(D_{N, k'})$, for $l\in I\cup J, k\in K, k'\in K'$. Then
$$X_N^{l,k}I_N=Y_N^{l,k}=Z_N^{l,k}I_N, X_N^{k'}I_N=D_{N, k'}=Z_N^{k'}, \forall l\in I\cup J, k\in K, k'\in K'. $$ Moreover, by Remark 3.2 in \cite{PS}, $L(A)R(B)=R(B)L(A)$, for $A, B\in M_N(\L(\X))$. By Theorem 2.6 in this paper and Theorem 4.13 in \cite{PS}, $((X_N^{l,k})_{l\in I\cup J}, (Z_N^{l,k})_{l\in I\cup J}), k\in K, ((X_N^{k'})_{k'\in K'}, (Z_N^{k'})_{k'\in K'})$ are asymptotic bi-free. It implies that $((L(Y_N^{i,k}))_{i\in I}, (R(Y_{N}^{j,k}))_{j\in J}), k\in K, ((L(D_{N,k'})_{k'\in K'}, (R(D_{N, k'}))_{k'\in K'})$ are asymptotic bi-free.



\end{proof}

\begin{Remark} Theorem 2.6 and Corollary 2.8 deal with asymptotic (bi-)freeness of  random two-faced families  of matrices of non-Gaussian random variables from two-faced families of constant {\sl diagonal} matrices. We choose diagonal matrices for two reasons. One is that diagonal matrices are technically easy to be treated in computing joint moments. The other reason is that, theoretically, diagonal matrices are general enough to approximate any measure on $\mathbb{R}$: Let $\mu$ be a probability measure on $\mathbb{R}$ for which all moments exist. Then there exists a sequence $\{D_N\in M_N(\mathbb{C})\}$ of diagonal matrices such that $$\lim_{N\rightarrow \infty}\frac{1}{N}Tr(D_N^m)=\int_\mathbb{R}t^md\mu(t), \forall m\in \mathbb{N}$$ (Exercise 22.27 in \cite{NS}).

Theoretically, one can prove asymptotic bi-freeness of two-faced families in Corollary 2.8 from a two-faced family of constant matrices. But it could be more complicated  technically than the diagonal matrix case.
\end{Remark}

\end{document}